\pdfoutput=1

\documentclass[11pt]{article}
\usepackage[accsupp]{axessibility}
\usepackage[latin1]{inputenc}

\usepackage{pgf}
\usepackage{amsmath,amsfonts,amssymb,amscd,amsthm,graphicx}

\usepackage{color}
\usepackage{float}
\usepackage{sectsty}

\usepackage{amsfonts}
\usepackage{latexsym}

\usepackage{amsmath}
\usepackage{amsthm}
\usepackage{makeidx}
\usepackage{latexsym}
\usepackage{graphicx}
\usepackage{amssymb}

\usepackage{hyperref}
\usepackage{float}
\newtheorem{theorem}{Theorem}

\theoremstyle{remark}
\newtheorem{remark}{\emph{Remark}}
\theoremstyle{remark}
\newtheorem{definition}{Definition}
\theoremstyle{remark}

\usepackage[latin1]{inputenc}

\title{Approximations of algebraic irrationalities with matrices}

\author{Stefano Barbero, Umberto Cerruti, Nadir Murru \\ University of Turin, Department of Mathematics\\ Via Carlo Alberto 10, Turin, Italy \\ stefano.barbero@unito.it, umberto.cerruti@unito.it, nadir.murru@unito.it}

\date{}

\begin{document}

\maketitle

\begin{abstract}

We discuss the use of matrices for providing sequences of rationals that approximate algebraic irrationalities. In particular, we study the regular representation of algebraic extensions, proving that ratios between two entries of the matrix of the regular representation converge to specific algebraic irrationalities. As an interesting special case, we focus on cubic irrationalities giving a generalization of the Khovanskii matrices for approximating cubic irrationalities. We discuss the quality of such approximations considering both rate of convergence and size of denominators. Moreover, we briefly perform a numerical comparison with well--known iterative methods (such as Newton and Halley ones), showing that the approximations provided by regular representations appear more accurate for the same size of the denominator.

\end{abstract}



\section{Convergence properties for regular representations of algebraic extensions}
\label{regular}
\normalsize
The study of approximations of irrational numbers by means of rationals is a very important and rich research field. This research field is named \emph{Diophantine approximation} in honor of Diophantus of Alexandria whose studies principally had dealt with researching rational solutions of algebraic equations. During the years, mathematicians have considerably improved results about Diophantine approximation.

In this context iterative methods, such as Newton method and higher order generalizations (i.e., Householder methods \cite{Householder}) are widely used and studied. Recently, many different iterative root--finding methods have been developed improving classical methods (see, e.g., \cite{Abba}, \cite{Noor}, \cite{Grau}). However, the iterative methods are computationally slow and denominator size of the provided rational approximations rapidly increases. On the other hand, continued fractions provide best approximations of real numbers. However, their use is not ever convenient from a computational point of view.

In the case of algebraic numbers, iterative methods can be replaced by more convenient ones. For example, in \cite{Pet}, the authors propose an algorithm based on the LLL-reduction procedure for approximating algebraic numbers. Recently, different techniques involving powers of \(2\times 2\) matrices have been developed for approximating quadratic irrationalities (see, e.g., \cite{Wild} and \cite{abcm}). In \cite{Khov} and \cite{Lau}, authors introduced particular \(3\times 3\) matrices for studying approximations of cubic roots. The use of matrices is very advantageous since power of matrices can be fastly evaluated and their entries are linear recurrent sequences whose properties can be exploited to study convergence. Moreover, study of simultaneous approximations is a very classical and well investigated topic, see, e.g., \cite{Adams} and \cite{Che}.

In the following, firstly, we introduce a family of matrices starting from the regular representation of algebraic extensions, studying their approximating properties. Then in section \ref{gen}, we focus on cubic irrationalities, generalizing Khovanskii matrices and other kinds of matrices used in the approximation of cubic irrationalities. Moreover, In section \ref{num}, we provide numerical results about the studied approximations. In particular, we discuss performances of our approximations with respect to some parameters and we compare them with certain well--known iterative methods, such as Newton, Halley, and Noor methods.

Let \(\alpha\) be a real root of \(f(t) = t^m-\sum\sb{s=0}^{m-1} u\sb{m-s}t^s\), with \(u\sb i\in\mathbb Q\), for \(i=1,...,m\), irreducible over \(\mathbb Q\). The algebraic extension \(\mathbb Q(\alpha)\) has basis \((1,\alpha,\alpha^2,...,\alpha^{m-1})\).
 Let \(\sum\sb{i=0}^{m-1} x\sb i\alpha^i\) be an element of \(\mathbb Q(\alpha)\), it can be represented by the \(m\times m\) matrix \(M=(M\sb{i,j})\)such that

\begin{equation} \label{eq:rep}\sum\sb{i=0}^{m-1} x\sb i\alpha^i\alpha^{j-1}=\sum\sb{i=1}^mM\sb{i,j}\alpha^{i-1}, \quad j=1,...,m.\end{equation}

The matrix \(M\) is usually called the regular representation of \(\mathbb Q(\alpha)\). Let us observe that the above identities can be written also in the case that \(f(t)\) is reducible. Thus, in the following, we do not restrict \(f(t)\)to be irreducible and we formally define the matrix \(M\) by means of \eqref{eq:rep}. Sometimes we will use the notation \(M(\textbf{x}, \textbf{u})\), where \(\textbf{x}=(x\sb0,...,x\sb{m-1})\) and \(\textbf{u} = (u1,...,u\sb m)\).

Let \(A\) be the companion matrix of \(f(t)\) defined as
\begin{equation*}
A=\begin{bmatrix}
0 & 0 & 0 & \cdots & 0 & 0 & u\sb{m}\\
1 & 0 & 0 & \cdots & 0 & 0 & u\sb{m-1}\\
0 & 1 & 0 & \cdots & 0 & 0 & u\sb{m-2}\\
\cdots & \cdots & \ddots & \cdots & \cdots & \cdots & \cdots\\
0 & 0 & 0 & 1 & 0 & 0 & u\sb{3}\\
0 & 0 & 0 & 0 & 1 & 0 & u\sb{2}\\
0 & 0 & 0 & 0 & 0 & 1 & u\sb{1}
\end{bmatrix}\label{eq:companionm}
\end{equation*}

By definition of \(M\), it follows that \(M = \sum\sb{n=0}^{m-1} x\sb nA^n\). The entries of matrices \(A^n\) can be explicitly written, see, e.g., Theorem 3.1 in \cite{Chen} (note that here the companion matrix is written in a slightly different form). In this way, the entries of the matrix \(M\) have the explicit expression \(M\sb{i,j} = \sum\sb{n=0}^{m-1} x\sb na\sb{i,j}^{(n)}\), where
\begin{equation*}
a\sb{i,j}^{\left(n\right)}=\underset{k\sb{1}+2k\sb{2}+\cdots+mk\sb{m}=n-i+j}{\sum}\frac{k\sb{m+1-i}+k\sb{m+2-i}+\cdots+k\sb{m}}{k\sb{1}+k\sb{2}+\cdots+k\sb{m}}\left(\begin{array}{c}
k\sb{1}+k\sb{2}+\cdots+k\sb{m}\\
k\sb{1},\ldots,k\sb{m}
\end{array}\right)u\sb{1}^{k\sb{1}}u\sb{2}^{k\sb{2}}\cdots u\sb{m}^{k\sb{m}}\label{eq:anij}
\end{equation*}
for \(k\sb1,...,k\sb m\) non--negative integers and \(\begin{pmatrix}
k\sb{1}+k\sb{2}+\cdots+k\sb{m}\\
k\sb{1},\ldots,k\sb{m}
\end{pmatrix}\) is the multinomial coefficient.

In the following theorem we show convergence properties of \(M\) by means of the Vandermonde matrix.

\begin{theorem} \label{thm:main}
Let \(\alpha\sb1, ..., \alpha\sb m\) be distinct roots of \(f(t) = t^m-\sum\sb{s=0}^{m-1} u\sb{m-s}t^s\). Let \(V = V(\alpha\sb1, ..., \alpha\sb m)\) and \(M(\textbf{x}, \textbf{u})\) be respectively the Vandermonde matrix of \(f(t)\) and the matrix defined by \eqref{eq:rep}, with \(\alpha = \alpha\sb k\) for a given \(k\in \{1,\ldots,m\} \). Let us define 

\[c(\textbf{x},\alpha\sb k) = \min \left\lbrace \cfrac{\lvert \sum\sb{i=0}^{m-1}x\sb i\alpha^i\sb k \rvert}{\lvert \sum\sb{i=0}^{m-1}x\sb i\alpha^i\sb j \rvert}: j = 1, ..., m, \alpha\sb j\not=\alpha\sb k \right\rbrace.\]
\noindent
If \(c(\textbf{x},\alpha\sb k)>1\), then 

\[\lim\sb{n\rightarrow+\infty}\cfrac{M^n\sb{i,j}}{M^n\sb{p,q}}=\cfrac{V^{-1}\sb{i,k}V\sb{k,j}}{V^{-1}\sb{p,k}V\sb{k,q}},\]

given any index \(i,j,p,q\in\{1,...,m\}\) (such that \(i \not= p\) and/or \(j \not= q\)).

\end{theorem}

\begin{proof}
Let \(V=V(\alpha\sb1, ..., \alpha\sb m)\) be the Vandermonde matrix of \(f(t)\), i.e.,

\begin{equation*}
V = \left[\begin{array}{ccccc}
1 & \alpha\sb1 & \alpha\sb1^2 & \cdots & \alpha\sb1^{m-1}\\
1 & \alpha\sb2 & \alpha\sb2^2 & \cdots & \alpha\sb2^{m-1}\\
\vdots & \vdots & \vdots & \cdots & \vdots\\
1 & \alpha\sb m & \alpha\sb m^2 & \cdots & \alpha\sb m^{m-1}
\end{array}\right]
\end{equation*}

It is well--known that \(V\) can be used in order to diagonalize the companion matrix of \(f(t)\) (see, e.g., \cite{Lan} pag. 69), i.e.,

\begin{equation*}
D(\alpha\sb1, ..., \alpha\sb m):=
\left[\begin{array}{cccc}
\alpha\sb1 & 0 & \cdots & 0\\
0 & \alpha\sb2 & \cdots & 0\\
\vdots & \vdots & \ddots & \vdots\\
0 & 0 & \cdots & \alpha\sb m
\end{array}\right] = VAV^{-1}.
\end{equation*}

Since \(M\) can be written as a linear combination of powers of the companion matrix, we have

\[VMV^{-1}=D\left(\sum\sb{i=0}^{m-1}x\sb i\alpha^i\sb1, ..., \sum\sb{i=0}^{m-1}x\sb i\alpha^i\sb m\right)\]

and

\[VM^nV^{-1}=D\left(\left(\sum\sb{i=0}^{m-1}x\sb i\alpha^i\sb 1\right)^n, ..., \left(\sum\sb{i=0}^{m-1}x\sb i\alpha^i\sb m\right)^n\right).\]

From the previous identity, we obtain

\[M\sb{i,j}^n=V^{-1}\sb{i,1}V\sb{1,j}\left(\sum\sb{i=0}^{m-1}x\sb i\alpha^i\sb1\right)^n+...+V^{-1}\sb{i,m}V\sb{m,j}\left(\sum\sb{i=0}^{m-1}x\sb i\alpha^i\sb m\right)^n\]

and finally

\[\cfrac{M^n\sb{i,j}}{M^n\sb{p,q}}=\cfrac{V^{-1}\sb{i,1}V\sb{1,j}\left(\sum\sb{i=0}^{m-1}x\sb i\alpha^i\sb1\right)^n+...+V^{-1}\sb{i,m}V\sb{m,j}\left(\sum\sb{i=0}^{m-1}x\sb i\alpha^i\sb m\right)^n}{V^{-1}\sb{p,1}V\sb{1,q}\left(\sum\sb{i=0}^{m-1}x\sb i\alpha^i\sb1\right)^n+...+V^{-1}\sb{p,m}V\sb{m,q}\left(\sum\sb{i=0}^{m-1}x\sb i\alpha^i\sb m\right)^n}\]
from which the thesis easily follows dividing numerator and denominator by \(\left(\sum\sb{i=0}^{m-1}x\sb i\alpha^i\sb k\right)^n\).
\end{proof}

\begin{remark}
The condition \(c(\textbf{x},\alpha\sb k)>1\) in the above theorem allows to specify which is the limit of the ratio \(\cfrac{M^n\sb{i,j}}{M^n\sb{p,q}}\), given any index \(i,j,p,q\in\{1,...,m\}\). In particular, it specifies which root \(\alpha\sb{k}\) of the polynomial \(f(t) = t^m-\sum\sb{s=0}^{m-1} u\sb{m-s}t^s\) can be approximated using the matrix \(M\), since \(V^{-1}\sb{i,k}V\sb{k,j}\) and \(V^{-1}\sb{p,k}V\sb{k,q}\), given any index \(i,j,p,q\in\{1,...,m\}\), are all quantities that can be written involving only the root \(\alpha\sb k\), its powers and coefficients of \(f(t)\), as direct calculations on the entries of the Vandermonde matrix show. In section \ref{gen}, we will see some examples regarding cubic polynomials. In the case of totally real cubic and irreducible polynomials, we can see that there always exists an index \(k\) such that the condition \(c(\textbf{x},\alpha\sb k)>1\) holds. Thus, when the quantity \(\textbf{x}\) is fixed, we can check the values of \(c(\textbf{x},\alpha\sb k)\), for any index \(k\), in order to know which specific root of \(f(t)\) we approximate using \(M\) with the chosen value of \(\textbf{x}\). On the other hand, we can choose the values of \(\textbf{x}\) so that  \(c(\textbf{x},\alpha\sb k) > 1\) for a specific root \(\alpha\sb k\) that we would like to approximate.
\end{remark}

We study the rate of convergence of the ratios of the entries of \(M^n\) in the next theorem, where we will use the following notation:
\[L = \cfrac{A\sb k}{B\sb k} = \cfrac{V\sb{i,k}^{-1} V\sb{k,j}}{V\sb{p,k}^{-1}V\sb{k,q}},\]
with \(V\) the Vandermonde matrix as defined in Theorem \ref{thm:main}. Moreover, we will consider
\[\gamma\sb j = \sum\sb{i=0}^{m-1} x\sb i \alpha\sb j^i,\]
and
\[c^{-1}(\textbf{x},\alpha\sb k) = \max \left\lbrace \cfrac{\lvert \sum\sb{i=0}^{m-1}x\sb i\alpha^i\sb j \rvert}{\lvert \sum\sb{i=0}^{m-1}x\sb i\alpha^i\sb k \rvert}: j = 1, ..., m, \alpha\sb j\not=\alpha\sb k \right\rbrace = \max \left\lbrace \cfrac{\lvert \gamma\sb j \rvert}{\lvert \gamma\sb k \rvert}: j = 1, ..., m, \alpha\sb j\not=\alpha\sb k \right\rbrace,\]
where \(\gamma\sb j^n\)'s are the roots of the characteristic polynomial of \(M^n\).

\begin{theorem}
Let \(\alpha\sb 1, ..., \alpha\sb m\) be distinct roots of \(f(t) = t^m-\sum\sb{s=0}^{m-1} u\sb{m-s}t^s\). Let \(V = V(\alpha\sb1, ..., \alpha\sb m)\) and \(M(\textbf{x}, \textbf{u})\) be respectively the Vandermonde matrix of \(f(t)\) and the matrix defined by \eqref{eq:rep}, with \(alpha = \alpha\sb k\) for a given  \(k\in \{1,\ldots, m\}\). If \(c(\textbf{x},\alpha\sb k) > 1\) (i.e. \(0 < c^{-1}(\textbf{x},\alpha\sb k) < 1\)) and \(A\sb lB\sb k - A\sb kB\sb l\not= 0\) (where \(l\) is the index such that \(c^{-1}(\textbf{x},\alpha\sb k) = \cfrac{|\gamma\sb l|}{|\gamma\sb k|}\)), then the order of convergence of \(\left| \cfrac{M^n\sb{i,j}}{M^n\sb{p,q}} - L \right|\) is \(O\left((c^{-1}(\textbf{x},\alpha\sb k))^n\right)\). 
\end{theorem}

\begin{proof}
From the previous theorem, we know that
\[\cfrac{M\sb{i,j}^n}{M\sb{p,q}^n} = \cfrac{\sum\sb{s=1}^m A\sb s \gamma\sb s^n}{\sum\sb{s=1}^m B\sb s \gamma\sb s^n} = \cfrac{A\sb k + \sum\sb{s=1, s\not=k}^m A\sb s \left(\frac{\gamma\sb s}{\gamma\sb k}\right)^n}{B\sb k + \sum\sb{s=1, s\not=k}^m B\sb s \left(\frac{\gamma\sb s}{\gamma\sb k}\right)^n}.\]
Let \(l\) be the index such that \(c^{-1}(\textbf{x},\alpha\sb k) = \cfrac{|\gamma\sb l|}{|\gamma\sb k|},\) settingn\( y = \cfrac{\gamma\sb l}{\gamma\sb k}\), we have
\[\cfrac{M\sb{i,j}^n}{M\sb{p,q}^n} = B\sb k^{-1}\left(A\sb k + A\sb l y^n + o (y^n)\right)\left(1 + \frac{B\sb l}{B\sb k}y^n+o(y^n)\right)^{-1} =\]  \[=B\sb k^{-1}\left(A\sb k - \frac{A\sb kB\sb l}{B\sb k}y^n + A\sb ly^n + o(y^n)\right)= \cfrac{A\sb k}{B\sb k} + \cfrac{A\sb lB\sb k - A\sb kB\sb l}{B\sb k^2} y^n + o(y^n).\]
Hence, we obtain

\[\left| \cfrac{M\sb{ij}^n}{M\sb{pq}^n} - L \right| = \left| \cfrac{A\sb lB\sb k - A\sb kB\sb l}{B\sb k^2} y^n + o(y^n) \right| \leq \left|  \cfrac{A\sb lB\sb k - A\sb kB\sb l}{B\sb k^2} \right| |y|^n + o(|y|^n) \leq C \left(c^{-1}(\textbf{x},\alpha\sb k)\right)^n,\]

for a certain constant \(C\) depending on \(\left|  \cfrac{A\sb lB\sb k - A\sb kB\sb l}{B\sb k^2} \right|\).
\end{proof}

\begin{remark} \label{rem:thm2}
The above theorem holds under the condition  \(A\sb lB\sb k - A\sb kB\sb l\not= 0\). We can observe that the case \(A\sb lB\sb k - A\sb kB\sb l = 0\), or equivalently \(\frac{A\sb{k}}{B\sb{k}}=\frac{A\sb{l}}{B\sb{l}}\) may occur in some special situations. In particular, there are some ratios of elements of \(M^n\) that are constant quantities. In fact, the following equalities hold for all \(t\)
\[ V\sb{m,t}^{-1} = \cfrac{1}{\displaystyle \prod\sb{\substack { 1 \leq h \leq m \\ h\not=t }}{(\alpha\sb t - \alpha\sb h)}},
  \quad V\sb{t,m-1} = \alpha\sb t^{m-2} \]
 \[V\sb{1,t}^{-1} = (-1)^{m-1} \cfrac{\displaystyle \prod\sb{i=1, i\not=t}^m \alpha\sb i}{\displaystyle \prod\sb{\substack { 1 \leq h \leq m \\ h\not=t }}{(\alpha\sb t - \alpha\sb h)}}, \quad V\sb{t,m} = \alpha\sb t^{m-1},
\]
thus we obtain that the ratios 
\[\cfrac{A\sb t}{B\sb t} = \cfrac{1}{(-1)^{m-1}\prod\sb{h=1}^m \alpha\sb h},\]
are clearly independent from the choice of \(t\)i.e., we have \(A\sb lB\sb k - A\sb kB\sb l = 0\) when \((i,j,p,q)=(m,m-1,1,m)\), and, for all \(n\geq 1\), the ratios \(\frac{M^n\sb{m,m-1}}{M^n\sb{1,m}}\) have the same value. We also have a similar situation when \((i,j,p,q)=(m,1,1,2)\), i.e., for all \(n\geq 1\), the ratios
\(\cfrac{M\sb{m,1}^{n}}{M\sb{1,2}^{n}}\) are constant quantities again, since  we obtain  for any index \(t\)
\[\cfrac{A\sb t}{B\sb t} = \cfrac{1}{(-1)^{m-1}\prod\sb{h=1}^m \alpha\sb h}.\]
\end{remark}

In the next section, we focus on the cubic case, since some well--known and studied matrices arise as particular cases of the matrix \(M\).

\section{Approximations of cubic irrationalities}
\label{gen}
The following matrix
\[\begin{pmatrix} x & r \cr 1 & x  \end{pmatrix},\]
for \(x,r\in\mathbb Z\) and \(r\) positive square--free, is used to determine classic R\'{e}dei rational functions \cite{Redei}. Powers of this matrix yield rational approximations of \(\sqrt{r}\). In \cite{abcm}, the authors proved that among these approximations, Pad\'{e} and Newton approximations can be found. A natural generalization of this matrix is given by 
\[A=\begin{pmatrix}  x & r & r \cr 1 & x & r \cr 1 & 1 & x \end{pmatrix}\]
for \(x,r\in\mathbb Z\) and \(r\) cube--free. This matrix has been introduced by Khovanskii \cite{Khov} to approximate \(\sqrt[3]{r}\) and \(\sqrt[3]{r^2}\). Let \(A^n\sb{i,j}\) denote the \(i,j\)--th entry of \(A^n\), we have
\[\lim\sb{n\rightarrow+\infty}\cfrac{A^n\sb{1,1}}{A^n\sb{3,1}}=\sqrt[3]{r^2},\quad \lim\sb{n\rightarrow+\infty}\cfrac{A^n\sb{2,1}}{A^n\sb{3,1}}=\sqrt[3]{r}.\]
In \cite{Lau}, authors studied the role of \(x\) in order to ensure the fastest convergence. In \cite{abcm2}, authors focused on
\[B=\begin{pmatrix} x & r & 0 \cr 0 & x & r \cr 1 & 0 & x  \end{pmatrix}\]
and similarly, we have 
\[\lim\sb{n\rightarrow +\infty}\cfrac{B^n\sb{1,1}}{B^n\sb{3,1}}=\sqrt[3]{r^2},\quad \lim\sb{n\rightarrow+\infty}\cfrac{B^n\sb{2,1}}{B^n\sb{3,1}}=\sqrt[3]{r}.\]
Moreover, in this case the authors proved that \(\cfrac{B^n\sb{1,1}}{B^n\sb{3,1}}\) and \(\cfrac{B^n\sb{2,1}}{B^n\sb{3,1}}\) are convergents of certain generalized continued fractions yielding periodic representations of cubic roots. Finally, let \(\alpha\) be the real root largest in modulus of \(t^3-pt^2-qt-r\), with \(p,q,r\in\mathbb Q\). In \cite{m}, the author showed that matrix
\[ C=\begin{pmatrix} x & r & pr \cr 0 & q+x & pq+r \cr 1 & p & p^2+q+x \end{pmatrix}\]
yields simultaneous rational approximations of \(\alpha-p\) and \(\cfrac{r}{\alpha}\). However, in \cite{m} the author did not focus on the study of rational approximations, but studied matrix \(C\) in order to determine periodic representations for any cubic irrational.

Matrices \(A\), \(B\), and \(C\) are all particular cases of the matrix \(M\) studied in the previous section. Indeed, if we consider the cubic polynomial \(t^3-pt^2-qt-r\), \(p,q,r\in \mathbb Q\), then, given integer numbers \(x,y,z\), we have

\begin{equation} \label{mir}
M((x,y,z),(p,q,r)) = 
\left[\begin{array}{ccc}
x & rz & ry+prz\\
y & x+qz & qy+(pq+r)z\\
z & y+pz & x+py+(p^2+q)z
\end{array}\right].
\end{equation}

Previous matrices are particular cases of \(M\). Indeed,
\[A=M((x,1,1),(0,0,r)),\quad B=M((x,0,1),(0,0,r)),\quad C=M((x,0,1),(p,q,r)).\]

By Theorem \ref{thm:main}, it is possible to explicitly write limits of ratios between two elements of \(\left(M((x,y,z),(p,q,r))\right)^n\). Let \(\alpha\sb1, \alpha\sb2, \alpha\sb3\) be roots of \(t^3-pt^2-qt-r\) and suppose \(c((x,y,z),\alpha\sb1)>1\), i.e., we have chosen \(x, y, z\) so that matrix \(M\) can be used for approximating \(\alpha\sb1\). For instance, we have
\[\lim\sb{n\rightarrow+\infty}\cfrac{M^n\sb{2,2}}{M^n\sb{h,k}}=\bar M\sb{h,k},\]
given any \(h=1,2,3\) and \(k=1,2,3\), where
\[\bar M=\begin{pmatrix} -\cfrac{\alpha\sb1(\alpha\sb2+\alpha\sb3)}{\alpha\sb{2}\alpha\sb{3}} & -\cfrac{\alpha\sb2+\alpha\sb3}{\alpha\sb2\alpha\sb3} & -\cfrac{\alpha\sb2+\alpha\sb3}{\alpha\sb1\alpha\sb2\alpha\sb3} \cr \alpha\sb1 & 1 & \cfrac{1}{\alpha\sb1} \cr -\alpha\sb1(\alpha\sb2+\alpha\sb3) & -\alpha\sb2-\alpha\sb3 & -\cfrac{\alpha\sb2+\alpha\sb3}{\alpha\sb1} \end{pmatrix}=\begin{pmatrix} \cfrac{\alpha^{2}\sb1(\alpha\sb1-p)}{r} & \cfrac{\alpha\sb1(\alpha\sb1-p)}{r} & \cfrac{\alpha\sb1-p}{r} \cr \alpha\sb1 & 1 & \cfrac{1}{\alpha\sb1} \cr \alpha\sb1(\alpha\sb1-p) & \alpha\sb1-p & \cfrac{\alpha\sb1-p}{\alpha\sb1} \end{pmatrix}.\]
Another example is provided by
\[\lim\sb{n\rightarrow+\infty}\cfrac{M^n\sb{3,3}}{M^n\sb{h,k}}=\bar M\sb{h,k},\]
given any \(h=1,2,3\) and \(k=1,2,3\), we have
\[\bar M=\begin{pmatrix} \cfrac{\alpha\sb1^2}{\alpha\sb2\alpha\sb3} & \cfrac{\alpha\sb1}{\alpha\sb2\alpha\sb3} & \cfrac{1}{\alpha\sb2\alpha\sb3} \cr
 & & \cr -\cfrac{\alpha\sb1^2}{\alpha\sb2+\alpha\sb3} & -\cfrac{\alpha\sb1}{\alpha\sb2+\alpha\sb3} & -\cfrac{1}{\alpha\sb2+\alpha\sb3} \cr
 & & \cr \alpha\sb1^2 & \alpha\sb1 & 1 \end{pmatrix}=\begin{pmatrix} \cfrac{\alpha\sb1^3}{r} & \cfrac{\alpha\sb1^2}{r} & \cfrac{\alpha\sb1}{r} \cr & & \cr \cfrac{\alpha\sb1^2}{\alpha\sb1-p} & \cfrac{\alpha\sb1}{\alpha\sb1-p} & \cfrac{1}{\alpha\sb1-p} \cr & & \cr \alpha\sb1^2 & \alpha\sb1 & 1 \end{pmatrix}.\]
Clearly, if we have \(c((x,y,z),\alpha\sb2)>1\), then above results still hold exchanging indexes 1 and 2.

\begin{remark}
We can check the considerations of the Remark \ref{rem:thm2} in the above matrix. Indeed, we have that \(\cfrac{\bar M\sb{3,1}}{\bar M\sb{1,2}} = \cfrac{\bar M\sb{3,2}}{\bar M\sb{1,3}} = r\). 
\end{remark}

\section{Numerical results} \label{num}
In this section, we will deal with the quality of approximations provided by \(M\) comparing it with known iterative methods as Newton, Halley and similar ones. We would like to highlight that our method consists in evaluating powers of the matrix \(M\) and this is accomplished using only integer arithmetic, i.e., it is an error--free method.

\noindent
It is well--known that continued fractions provide best approximations of real numbers (see \cite{Olds} for a good survey about continued fractions). In particular, given the \(n\)--th convergent \(\cfrac{p\sb n}{q\sb n}\) of the continued fraction of a real number \(\alpha\), then 

\[\left| \alpha-\cfrac{p\sb n}{q\sb n} \right|\leq\left| \alpha-\cfrac{a}{b} \right|,\quad a,b\in\mathbb Z,\]
for all \(b\leq q\sb n\). However, evaluating approximations by means of continued fractions is not generally an used method since a continued fraction is a non--terminating expression. Indeed, many different methods are studied and used in this context. In the particular case of approximations of algebraic numbers, many root--finding algorithms have been developed. Here, we compare some of these methods with approximations provided by \(M\). Taking into account classic definition of best approximations above described, we will compare quality of rational approximations, provided by different methods, having denominators with same size.

We will consider real roots of the Ramanujan cubic polynomial \(t^3+t^2-2t-1\) and we study approximations of \(M\) for different values of \((x,y,z)\). The roots of this polynomial are quite famous (see, e.g., \cite{Wit}) and they are
\[\alpha\sb 1=2\cos\cfrac{2\pi}{7}, \quad \alpha\sb2=2\cos\cfrac{4\pi}{7},\quad \alpha\sb3=2\cos\cfrac{8\pi}{7}\]
with \(\lvert\alpha\sb3\rvert>\lvert\alpha\sb1\rvert>\lvert\alpha\sb2\rvert\). 

\subsection{Approximations of M for different values of \((x,y,z)\)} \label{xyz}
Considering \(M=M((x,y,z),(-1,2,1))\), for some values of \((x,y,z)\), we provide approximations of \(\alpha\sb3\) by means of the sequence \[m\sb n(x,y,z):=\cfrac{M^n\sb{2,1}(x,y,z)}{M^n\sb{3,1}(x,y,z)}-1.\]
In Table \ref{table:approx-m-a3} and \ref{table:approx-size-m}, we summarize quality of our approximations for different values of \((x,y,z)\). In particular, we consider
\[\begin{cases} x=0, y=0, z=1, \quad c(0,0,1)=2.08815 \cr
x=1, y=-1, z=1,\quad c(1,-1,1)=3.68141 \cr
x=0, y=-1, z=1,\quad c(0,-1,1)=7.85086 \cr
x=69, y=99, z=-124,\quad c(69,99,-124)=1343.4 \end{cases}\]
and we show values of \(\lvert m\sb n(x,y,z)-\alpha\sb3 \rvert\) and size of denominators of \(m\sb n(x,y,z)\), i.e., number of digits \(D\sb n(x,y,z)\) of \(M\sb{3,1}^n(x,y,z)\).

In Figures \ref{fig:approx} and \ref{fig:digits}, we depict the situations described in Tables \ref{table:approx-m-a3} and \ref{table:approx-size-m}, respectively.

We can observe that approximations provided by \(m\sb n(69,99,-124)\) are the most accurate. 
However, they have the greatest denominators. Thus these approximations could not be optimal taking into account previous considerations about continued fractions.

It is interesting to observe that approximations \(m\sb n(0,-1,1)\) are more accurate than \(m\sb n(1,-1,1)\) and furthermore they have smaller denominators. Thus, approximations \(m\sb n(0,-1,1)\) are surely better than approximations \(m\sb n(1,-1,1)\) in any case.

In Table \ref{table:comparison1}, we compare approximations whose denominators have the same number of digits (for 16, 35, and 62 digits). We can observe that approximations \(m\sb n(0,-1,1)\) are more accurate than others with the same size of denominators.   

In conclusion, if we want to obtain accurate approximations with low values of \(n\) and we are not interested in the size of denominators, it is sufficient to find values \((x,y,z)\) that maximize \(c(x,y,z,\alpha\sb3)\). However, in this way approximations could not be the better than others with the same size of denominators. Indeed, we have seen that approximations obtained in correspondence of \(c(0,-1,1,\alpha\sb3)=7.85086\) are more accurate than approximations with same size of denominators provided for \(c(69,99,-124,\alpha\sb3)=1343.4\). It would be really interesting to study techniques that allow to determine values of \((x,y,z)\) that provide best approximations in this sense.

\subsection{Approximations of M for same values of \((x,y,z)\)}

By Theorem \ref{thm:main}, we can obtain approximations of a cubic irrationality by using different ratios of \(M^n\), for same values of \((x,y,z)\). Let us consider \((x,y,z)\) such that \(c(x,y,z,\alpha\sb3)>1\), then the reader can check that
\[\lim\sb{n\rightarrow+\infty}\cfrac{M^n\sb{2,2}}{M^n\sb{2,1}}=\lim\sb{n\rightarrow+\infty}\cfrac{M^n\sb{2,3}}{M^n\sb{2,2}}=\lim\sb{n\rightarrow+\infty}\cfrac{M^n\sb{3,3}}{M^n\sb{3,2}}=\alpha\sb3. \]
In this paragraph, we briefly compare these approximations with each other. Let us consider \((x,y,z)=(0,-1,1)\). In Tables \ref{table:approx-same-xyz} and \ref{table:digits-same-xyz}, we report distance from exact value of \(\alpha\sb3\) and size of denominators for these approximations, respectively. We can see that there are not significative differences among these approximations. Similar results are obtained for other triples \((x,y,z)\).

\subsection{Approximations of M for different values of \((p,q,r)\)}
By Theorem \ref{thm:main}, we can also obtain approximations of a cubic irrationality using different values of \((p,q,r)\). In this paragraph, we focus on approximations of \(\alpha\sb2\). 

Considering \((p,q,r)=(-1,2,1)\) we can find a triple \((x,y,z)\) such that \(c(x,y,z,\alpha\sb2)>1\). For instance we have \(c(10,-2,-3,\alpha\sb2)=2.67\) and we know that
\[\lim\sb{n\rightarrow+\infty}\cfrac{M^n\sb{2,1}}{M^n\sb{3,1}}-1=\alpha\sb2.\]
Moreover, we can consider polynomial \(t^3+2t^2-t-1\) (i.e., the reflected polynomial of \(t^3+t^2-2t-1\)) whose roots are \(frac{1}{\alpha\sb1}\), \(\frac{1}{\alpha\sb2}\), \(\frac{1}{\alpha\sb3}\). In this case we use \((p,q,r)=(-2,1,1)\) and for \((x,y,z)=(-3,1,-1,1/\alpha\sb2)\) we obtain \(c(x,y,z,1/\alpha\sb2)=2.67\). Thus, by theorem \ref{thm:main}, we have
\[\lim\sb{n\rightarrow+\infty}\cfrac{M^n\sb{1,1}}{M^n\sb{3,1}}-1=\alpha\sb2,\]
since \(r=1\). Let us observe that we have searched for triple \((x,y,z)\) determining a value of \(c(x,y,z,1/\alpha\sb2)\) similar to the previous case.

Finally, we can also consider \((p,q,r)=(2,1,-1)\). In this case \(t^3-2t^2-t+1\) has roots \(\alpha\sb1+1\), \(\alpha\sb2+1\), \(\alpha\sb3+1\) and the reader can check that, e.g, \(c(x,y,z,\alpha\sb2+1)=2.67\) so that
\[\lim\sb{n\rightarrow+\infty}\cfrac{M^n\sb{2,1}}{M^n\sb{3,1}}+1=\alpha\sb2,\]
since \(p=2\).

In Figures \ref{fig:approx2} and \ref{fig:digits2}, we depict behavior of these approximations, considering differences with exact value of \(\alpha\sb2\) and size of denominators, respectively. Even in this case, there are not significative differences among these approximations. Thus, in general quality of approximations is heavily affected by values of \(c(x,y,z,\alpha)\).

\subsection{Comparison with known root--finding methods}
In this  paragraph, we compare approximations provided by \(M\) with Newton, Halley, and Noor \cite{Noor} methods. We briefly recall these methods.
\begin{definition}
The Newton method provides rational approximations of a real root \(\alpha\) of \(f(t)\) by means of the sequence of rational numbers \(x\sb n\) by the equation
\[x\sb{n+1}=x\sb n-\cfrac{f(x\sb n)}{f'(x\sb n)},\quad \forall n\geq0\]
with a suitable initial condition \(x\sb0\).
\end{definition}
\begin{definition}
The Halley method provides rational approximations of a real root \(\alpha\) of \(f(t)\) by means of the sequence of rational numbers \(x\sb n\) by the equation
\[x\sb{n+1}=x\sb n-\cfrac{2f(x\sb n)f'(x\sb n)}{2(f'(x\sb n))^2-f(x\sb n)f''(x\sb n)},\quad \forall n\geq0\]
with a suitable initial condition \(x\sb0\).
\end{definition}
\begin{definition}
The Noor method provides rational approximations of a real root \(\alpha\) of \(f(t)\) by means of the sequence of rational numbers \(x\sb n\) by the equations
\[\begin{cases} y\sb n=x\sb n-\cfrac{f(x\sb n)}{f'(x\sb n)}, \quad \forall n\geq0 \cr x\sb{n+1}=y\sb n-\cfrac{f(y\sb n)}{f'(y\sb n)}-\cfrac{(f(y\sb n))^2f''(y\sb n)}{2(f'(y\sb n))^3},\quad \forall n\geq0 \end{cases}\]
with a suitable initial condition \(\sb0\).
\end{definition}
In Table \ref{table:n-h-n}, we report approximations of \(\alpha\sb3\) provided by Newton, Halley, and Noor methods. In particular we report the size of denominators and the difference between these approximations and the exact value of the root.

Using notation of subsection \ref{xyz}, let us consider \((p,q,r)=(-1,2,1)\) and \((x,y,z)=(0,-1,1)\). Sequence \(m\sb n(x,y,z)\) approximates \(\alpha\sb3\). In particular, we have that \(\lvert m\sb{10} - \alpha\sb3 \rvert=2.7\times10^{-9}\) and number of digits of denominator is 7. Furthermore, we have
\[\lvert m\sb{25} - \alpha\sb3 \rvert=1.0\times10^{-22}, \quad \lvert m\sb{852} - \alpha\sb3 \rvert=8.3\times10^{-763},\]
where denominators have 17 and 595 digits, respectively. Thus, approximations \(m\sb n\), with same accuracy of iterative methods, have size of denominators much less than iterative methods. Equivalently, we can say that our approximations, having same size of denominators with respect to iterative methods, are much more accurate.

If we are only interested to have high accuracy in few steps, we can consider \((p,q,r)=(-1,2,1)\), \((x,y,z)=(69,99,-124)\), \(N=M^3\) and \(N\sb n=\cfrac{N^n\sb{2,1}}{N^n\sb{3,1}}-1\). In Table \ref{table:n}, we report quality of approximations \(N\sb n\) for \(n=1,...,6\). We can observe that in this case we reach high accuracy in few steps, with better performances than iterative methods.

Finally, we would like to observe that evaluation of powers of matrices is very fast from a computational point of view and it is faster than iterative methods.

\section{Conclusion}\label{con}

We have introduced and studied a family of matrices (which generalize known ones) whose powers yield rational approximations of algebraic irrationalities. These matrices depend on some parameters whose meaning has been deeply discussed. These parameters allow to obtain many different approximations for the same irrational, providing a very handy method that can be adjusted as necessary, in order to obtain the desired quality of approximation. Numerical results have been also presented in order to show effectiveness of our approach. Some questions should be deeper analyzed:
\begin{itemize}
\item study the role of \(\textbf{x}\) in the size of denominators;
\item explicitly determine maximum of \(c(\textbf{x},\alpha)\);
\item study of the quality of simultaneous approximations (as defined, e.g., in \cite{Adams}).
\end{itemize}

\section{Acknowledgments}

The authors thanks the anonymous referee for carefully reading the paper and for all the suggestions and comments that greatly improved it.

\newpage

\section{FIGURES AND TABLES}\label{fig}

\begin{table}[hp]\small 
\caption{Quality of approximations of matrices \(M\): \(\lvert m\sb n(x,y,z)-\alpha\sb3 \rvert\).}
\centering
\tabcolsep=0.15cm
\scalebox{0.82}{
\begin{tabular}{|c||c|c|c|c|c|c|}
\hline 
& \(n=5\) & \(n=20\) & \(n=35\) & \(n=50\) & \(n=75\) & \(n=100\)  \cr \hline \hline
\(\lvert m\sb n(0,0,1)-\alpha\sb3 \rvert\) & 0.06 & 9.8\(\times10^{-7}\) & 1.6\(\times10^{-11}\) & 2.5\(\times10^{-16}\) & 2.5\(\times10^{-24}\) & 2.6\(\times10^{-32}\)  \cr \hline
\(\lvert m\sb n(1,-1,1)-\alpha\sb3 \rvert\) & 0.002 & 1.2\(\times10^{-11}\) & 3.8\(\times10^{-20}\) & 1.2\(\times10^{-28}\) & 8.7\(\times10^{-43}\) & 6.1\(\times10^{-57}\)   \cr \hline
\(\lvert m\sb n(0,-1,1)-\alpha\sb3 \rvert\) & 8\(\times 10^{-5}\) & 3.1\(\times10^{-18}\) & 1.2\(\times10^{-31}\) & 4.4\(\times 10^{-45}\) & 1.9\(\times10^{-67}\) & 7.9\(\times10^{-90}\)  \cr \hline
\(\lvert m\sb n(69,99,-124)-\alpha\sb3 \rvert\) & 1\(\times10^{-15}\) & 4.0\(\times10^{-63}\) & 9.5\(\times10^{-110}\) & 8.6\(\times10^{-157}\) & 6.1\(\times10^{-235}\) & 3.7\(\times10^{-313}\) \cr \hline

\end{tabular}
}
\label{table:approx-m-a3}
\end{table}

\begin{table}[hp]\small 
\caption{Quality of approximations of matrices \(M\): number of digits \(D\sb n(x,y,z)\) of \(M\sb{31}^n(x,y,z)\).}
\centering
\tabcolsep=0.15cm
\scalebox{0.82}{
\begin{tabular}{|c||c|c|c|c|c|c|}
\hline 
\textbf{Number of digits}& \(n=5\) & \(n=20\) & \(n=35\) & \(n=50\) & \(n=75\) & \(n=100\)  \cr \hline \hline
\(D\sb n(0,0,1)\) & 2 & 9 & 14 & 25 & 36 & 49  \cr \hline
\(D\sb n(1,-1,1)\) & 4 & 16 & 21 & 39 & 59 & 78   \cr \hline
\(D\sb n(0,-1,1)\) & 3 & 12 & 21 & 35 & 50 & 69  \cr \hline
\(D\sb n(69,99,-124)\) & 13 & 52 & 92 & 135 & 203 & 269 \cr \hline

\end{tabular}
}
\label{table:approx-size-m}
\end{table}

\begin{table}[hp]\small 
\caption{Comparison of approximations \(m\sb n(x,y,z)\) with same number of digits of denominators, for different values of \((x,y,z)\).}
\centering
\tabcolsep=0.15cm
\scalebox{0.82}{
\begin{tabular}{|c|c||c|c|}
\hline 
\((x,y,z)\) & \(n\) & \(\lvert m\sb n(x,y,z) -\alpha\sb3 \rvert\) & \(D\sb n(x,y,z)\)  \cr \hline \hline
\((0,0,1)\) & \(37\) & \(3.6\times 10^{-12}\) & \(16\)  \cr \hline
\((1,-1,1)\) & \(20\) & \(1.2\times 10^{-11}\) & \(16\)   \cr \hline
\((0,-1,1)\) & \(23\) & \(6.4\times 10^{-21}\) & \(16\)  \cr \hline
\((69,99,-124)\) & \(6\) & \(1.0\times 10^{-19}\) & \(16\) \cr \hline \hline
\((0,0,1)\) & \(74\) & \(5.3\times 10^{-24}\) & \(35\)  \cr \hline
\((1,-1,1)\) & \(45\) & \(8.3\times 10^{-26}\) & \(35\)   \cr \hline
\((0,-1,1)\) & \(50\) & \(4.4\times 10^{-45}\) & \(35\)  \cr \hline
\((69,99,-124)\) & \(14\) & \(1.9\times 10^{-44}\) & \(35\) \cr \hline \hline
\((0,0,1)\) & \(128\) & \(2.9\times 10^{-41}\) & \(62\)  \cr \hline
\((1,-1,1)\) & \(82\) & \(9.4\times 10^{-47}\) & \(62\)   \cr \hline
\((0,-1,1)\) & \(91\) & \(8.9\times 10^{-82}\) & \(62\)  \cr \hline
\((69,99,-124)\) & \(23\) & \(3.7\times 10^{-72}\) & \(62\) \cr \hline

\end{tabular}
}
\label{table:comparison1}
\end{table}
\newpage
\begin{table}[hp]\small 
\caption{Distance from exact value of \(\alpha\sb3\) and different ratios of \(M\) with \((x,y,z)=(0,-1,1)\).}
\centering
\tabcolsep=0.15cm
\scalebox{0.82}{
\begin{tabular}{|c||c|c|c|c|c|c|}
\hline 
& \(n=5\) & \(n=20\) & \(n=35\) & \(n=50\) & \(n=75\) & \(n=100\)  \cr \hline \hline
\(\left| \cfrac{M^n\sb{2,1}}{M^n\sb{3,1}}-1-\alpha\sb3 \right|\) &  8.0\(\times 10^{-5}\) &3.1\(\times10^{-18}\) &1.2\(\times10^{-31}\) & 4.4\(\times 10^{-45}\) & 1.9\(\times10^{-67}\) & 7.9\(\times10^{-90}\)  \cr \hline
\(\left| \cfrac{M^n\sb{2,2}}{M^n\sb{2,1}}-\alpha\sb3 \right|\) & 5.0\(\times10^{-5}\) & 2.1\(\times10^{-18}\) & 8.1\(\times10^{-32}\) & 3.0\(\times10^{-45}\) & 1.3\(\times10^{-67}\) &  5.4\(\times10^{-90}\)  \cr \hline
\(\left|\cfrac{M^n\sb{2,3}}{M^n\sb{2,2}}-\alpha\sb3 \right|\) & 2.0\(\times10^{-5}\) & 5.3\(\times10^{-19}\) & 2.0\(\times10^{-32}\) & 7.5\(\times10^{-46}\) & 3.2\(\times10^{-68}\) & 1.3\(\times10^{-90}\)  \cr \hline
\(\left| \cfrac{M^n\sb{3,3}}{M^n\sb{3,2}}-\alpha\sb3 \right|\) & 2.1\(\times10^{-5}\) & 7.6\(\times10^{-19}\) & 2.9\(\times10^{-32}\) & 1.1\(\times10^{-45}\) & 4.6\(\times10^{-68}\) & 1.9\(\times10^{-90}\) \cr \hline

\end{tabular}
}
\label{table:approx-same-xyz}
\end{table}

\begin{table}[hp]\small 
\caption{Number of digits of \(M^n\sb{3,1}\), \(M^n\sb{2,1}\), \(M^n\sb{2,2}\), \(M^n\sb{3,2}\) with \((x,y,z)=(0,-1,1)\).}
\centering
\tabcolsep=0.15cm
\scalebox{0.82}{
\begin{tabular}{|c||c|c|c|c|c|c|}
\hline 
\textbf{Number of digits}& \(n=5\) & \(n=20\) & \(n=35\) & \(n=50\) & \(n=75\) & \(n=100\)  \cr \hline \hline
\(M^n\sb{31}\) & 3 & 12 & 21 & 35 & 50 & 69  \cr \hline
\(M^n\sb{21}\) & 3 & 14 & 25 & 35 & 50 & 70   \cr \hline
\(M^n\sb{22}\) & 4 & 14 & 25 & 35 & 53 & 70  \cr \hline
\(M^n\sb{32}\) & 4 & 14 & 25 & 34 & 53 & 70 \cr \hline

\end{tabular}
}
\label{table:digits-same-xyz}
\end{table}

\begin{table}[hp]\small 
\caption{Quality of approximations of iterative methods.}
\centering
\tabcolsep=0.15cm
\scalebox{0.82}{
\begin{tabular}{|c|c|c|c|}
\hline 
\textbf{Method} & \(n\) & n. digits den. & \(\lvert x\sb n - \alpha\sb3 \rvert\)  \cr \hline \hline
Newton & 3 & 9 & 1.1\(\times 10^{-6}\)  \cr \hline
Newton & 5 & 80 & 9.2\(times 10^{-14}\)  \cr \hline
Newton & 10 & 19352 & 3.7\(\times 10^{-762}\)  \cr \hline
Halley & 2 & 9 & 8.1\(\times 10^{-8}\)  \cr \hline
Halley & 3 & 45 & 4.8\(\times 10^{-22}\)  \cr \hline
Halley & 6 & 28140 & 1.2\(\times 10^{-527}\)  \cr \hline
Noor & 2 & 18 & 1.1\(\times 10^{-6}\)  \cr \hline
Noor & 3 & 186 & 2.7\(times 10^{-18}\)  \cr \hline
Noor & 6 & 43136 & 4.8\(\times 10^{-471}\)  \cr \hline

\end{tabular}
}
\label{table:n-h-n}
\end{table}
\begin{table}[hp]\small 
\caption{Quality of approximations of \(N\sb n\).}
\centering
\tabcolsep=0.15cm
\scalebox{0.82}{
\begin{tabular}{|c|c|c|}
\hline 
 \(n\) & n. digits den. & \(\lvert N\sb n - \alpha\sb3 \rvert\)  \cr \hline \hline
 1 & 8 & 1.9\(\times 10^{-9}\)  \cr \hline
 2 & 24 & 2.8\(\times 10^{-28}\)  \cr \hline
 3 & 73 & 1.1\(\times 10^{-84}\)  \cr \hline
 4 & 219 & 1.0\(\times 10^{-253}\)  \cr \hline
 5 & 658 & 1.\(\times 10^{-760}\)  \cr \hline
 6 & 1975 & 8.4\(\times 10^{-2281}\)  \cr \hline
\end{tabular}
}
\label{table:n}
\end{table}

%
%
%


\begin{thebibliography}{99}

\bibitem{Abba} Abbasbandy, S: Improving Newton--Raphson method for nonlinear equations by modified Adomian decomposition method, \emph{Applied Mathematics and Computation}, Vol. 145, 887--893, (2003).

\bibitem{abcm2} Abrate, M., Barbero, S., Cerruti, U., Murru, N.: Periodic representations for cubic irrationalities, \emph{The Fibonacci Quarterly}, Vol. 50, No. 3, 252--264, (2012).

\bibitem{abcm} Abrate, M., Barbero, S., Cerruti, U., Murru, N.: Periodic representations and rational approximations of square roots, Journal of Approximation Theory, Vol. 175, 83--90, (2013).

\bibitem{Adams} Adams, W. W.: Simultaneous Diophantine approximations and cubic irrationals, \emph{Pacific Journal of Mathematics}, Vol. 30, No. 1, 1--14, (1969).

\bibitem{Chen} W. Y. C. Chen, J. D. Louck, \emph{The combinatorial power of the companion matrix}, Linear Algebra and its Applications, Vol. \textbf{232}, 261--278, 1996.

\bibitem{Che} Chevallier, N.: Best simultaneous diophantine approixmations of some cubic algebraic numbers, \emph{Journal de théorie des nombres de Bordeaux}, Vol. 14, Issue 2, 403--414, (2002).

\bibitem{Grau} Grau, M., Diaz--Barrero, J. L.: An improvement to Ostrowski root--finding method, \emph{Applied Mathematics and Computation}, Vol. 173, 450--456, (2006).

\bibitem{Householder} Householder, A. S.: The Numerical Treatment of a Single Nonlinear Equation, McGraw--Hill, New York, 1970.

\bibitem{Khov} Khovanskii, A. N.: The Application of Continued Fractions and Their Generalizations to Problems in Approximation Theory, cap. xii, 1963.

\bibitem{Lan} Lancaster, P., Tismenetsky, M.: The theory of matrices, Academic Press, 1985.

\bibitem{Lau} Mc Laughin, J., Sury, B.: Some observations on Khovanskii's matrix methods for extracting roots of polynomials, \emph{Integers}, Vol. 7, Article A48, (2007).

\bibitem{m} Murru, N.: On the periodic writing of cubic irrationals and a generalization of R\'{e}dei functions, \emph{International Journal of Number Theory}, Vol. 11, No. 11, 779--799, (2015).

\bibitem{Noor} Noor, K. I., Noor, M. A., Momani, S.: Modified Householder iterative method for nonlinear equations, \emph{Applied Mathematics and Computation}, Vol. 190, 1534--1539, (2007).

\bibitem{Olds} Olds, C. D.: Continued fractions, Random House, 1963.

\bibitem{Pet} A. Peth\H{o}, M. E. Pohst, C. Bertòk, On multidimensional Diophantine approximation of algebraic numbers, \emph{Journal of Number Theory}, Vol. 171, 422--448, (2017).

\bibitem{Redei} R\'{e}dei, L.: \"{U}ber eindeuting umkehrbare Polynome in endlichen K\"{o}rper, \emph{Acta Sci. Math. (Szeged)}, 11, 85--92, (1946).

\bibitem{Wild} Wildberger, N. J.: Pell's equation without irrational numbers, \emph{Journal of Integer Sequences}, Vol. 13, Article 10.4.3, (2010).

\bibitem{Wit} Witula, R., Slota, D.: New Ramanujan--type formulas and quasi--Fibonacci numbers of order 7, \emph{Journal of Integer Sequences}, Vol. 10, Article 07.5.6, (2007).

\end{thebibliography}
\end{document}